\documentclass[12pt]{article}

\usepackage[a4paper,centering,bindingoffset=0cm,hmargin=2cm,vmargin=2cm,includeheadfoot]{geometry}
\usepackage[font=footnotesize,labelfont=bf,width=0.75\textwidth,labelsep=period]{caption}
\usepackage{amsmath}
\usepackage{amssymb,bbm,color}
\usepackage{latexsym}
\usepackage{graphicx}
\usepackage{color,soul}
\usepackage{enumitem}
\usepackage[english]{babel}
\usepackage{authblk}

\usepackage{aliascnt}
\usepackage[hyperref,amsmath,thmmarks]{ntheorem}

\newaliascnt{proposition}{lemma}

\aliascntresetthe{proposition}

\newaliascnt{theorem}{lemma}
\newtheorem{theorem}[theorem]{Theorem}
\aliascntresetthe{theorem}

\newaliascnt{assumption}{lemma}

\aliascntresetthe{assumption}

\newaliascnt{remark}{lemma}
\newtheorem{remark}[remark]{Remark}
\aliascntresetthe{remark}

\theoremstyle{nonumberplain}
\theoremseparator{:}
\theoremheaderfont{\normalfont\itshape}
\theorembodyfont{\normalfont}
\theoremsymbol{\ensuremath{\square}}
\newtheorem{proof}{Proof}

\usepackage[pdftex,colorlinks=true,linkcolor=blue,unicode,pdfpagelabels]{hyperref}
\let\RE\Re
\let\Re=\undefined
\DeclareMathOperator{\Re}{\RE e}
\let\IM\Im
\let\Im=\undefined
\DeclareMathOperator{\Im}{\IM m}
\newcommand{\R}{\mathbbm R}
\renewcommand{\C}{\mathbbm C}

\newcommand{\e}{\mathrm e}

\usepackage{bm}

\DeclareMathOperator{\curl}{\nabla \times}

\newcommand{\Hi}{\b H^{int}}
\newcommand{\Ei}{\b E^{int}}
\newcommand{\He}{\b H^{ext}}
\newcommand{\Ee}{\b E^{ext}}

\newcommand{\he}{h^{ext}}
\newcommand{\ee}{e^{ext}}
\newcommand{\n}{\b n}
\newcommand{\ta}{\bm{\tau}}

\renewcommand{\b}[1]{\ensuremath{\mathbf{#1}}} 

\allowdisplaybreaks

\usepackage{array,multirow}
\usepackage{hhline}
{
\usepackage[table]{xcolor}
\usepackage{authblk}

\begin{document}


\title{The electromagnetic scattering problem by a cylindrical doubly-connected domain at oblique incidence: the direct problem }

\author{Leonidas Mindrinos\thanks{leonidas.mindrinos@univie.ac.at}}

\affil{Computational Science Center, University of Vienna,  Austria}


\maketitle

 \begin{abstract}
 We consider the direct electromagnetic scattering problem of time-harmonic obliquely incident waves by a infinitely long, homogeneous and doubly-connected cylinder in three dimensions. We apply a hybrid integral equation method (combination of the direct and indirect methods)  
 and we transform the scattering problem to a system of singular and hypersingular integral equations. The well-posedness of the corresponding problem is proven. We use trigonometric polynomial approximations and we solve the system of the discretized integral operators by a collocation method. 
 \end{abstract}

\section{Introduction}\label{sec_intro}

The scattering problem of electromagnetic waves by a penetrable medium generates theoretical and numerical questions. Even if it is the direct (given the medium, compute the scattered wave) or the inverse (recover the medium from the far-field pattern) problem, the main and first question to ask is that of the unique solvability. Numerically, both problems can be solved using similar techniques but the nonlinearity and the ill-posedness of the inverse problem have to be taken into account. For a review on scattering theory for solving direct and inverse problems, we refer to the books \cite{CakCol06, ColKre13, ColKre13b, KirHet15}.

Since we use electromagnetic waves as incident fields the mathematical model is based on Maxwell's equations and the transmission conditions describe the continuity of the tangential components of the electric and magnetic fields. The three-dimensional problem can, however, be reduced to simpler problems for the Helmholtz equation given some assumptions on the incident illumination and the optical properties of the medium.

We specify the medium to be a infinitely long, penetrable cylinder embedded in a homogeneous dielectric medium. The incoming electromagnetic wave is a time-harmonic plane wave at oblique incidence (transverse magnetic polarized). This problem has been considered by many researchers from different fields because of its applications in industry and medical imaging, see for instance \cite{ErtRoj00, LucPanSche10, Roj88, SarSen90}. From a mathematical point of view, this scattering problem has also attracted considerable attention. Many methods have been considered for the numerical solution of this problem, see \cite{CanLee91, Tsa07, TsiAli07, WuLu08}, but only recently the well-posedness of the direct problem has been addressed \cite{GinMin16, NakWan13, WanNak12}.

In this work we extend the results of \cite{GinMin16} to the case of a doubly-connected cylinder. The interior simply connected domain is a perfect electric conductor. This setup is motivated by the analysis of the behavior of antennas and tubes. We consider the Leontovich impedance boundary condition on the inner boundary together with transmission conditions on the outer boundary. Following \cite{WanNak12}, we see that the three-dimensional scattering problem is reduced to a system of four
two-dimensional Helmholtz equations for the interior and the exterior electric and magnetic fields. The complication of the problem lies in the reformulated boundary conditions where the tangential derivatives of the fields appear. 

We consider a hybrid integral equation method \cite{KleMar88}, meaning a combination of the direct (Green's formulas) and the indirect (single-layer ansatz) methods. The method of boundary integral equations has been considered for solving both direct and inverse scattering problems in different regimes. For some recent applications we refer to
 \cite{BouTurDom16, CakKre13, ChaGinMin17, ChaIvaPro13, GinMin17, IvaLou16, WanNak12}. We transform the direct problem to a system of singular and hypersingular integral equations. This system is of Fredholm type (uniqueness of solution) and the Fredholm alternative theorem gives existence. We use the collocation method to solve numerically the system of integral equations and we consider the Maue's formula for reducing the hypersingularity of the normal derivative of the double-layer potential \cite{Kre95, Kre14}. 

The paper is organized as follows: In \autoref{sec_problem} we formulate the direct scattering problem and we gather the necessary equations and boundary conditions. The existence and uniqueness of solutions, using Green's formulas and the integral equation method, are proved in \autoref{sec_unique}. In the last section we present numerical examples with analytic solutions that justify the applicability of the proposed scheme.

\section{Formulation of the problem}\label{sec_problem}

In this work we consider the scattering of a time-harmonic electromagnetic wave by an infinitely long, penetrable  and doubly-connected cylinder in three dimensions. We assume that the cylinder $\Omega_{int} \subset \R^3$ is oriented parallel to the $z-$axis and that it is homogeneous, meaning its properties are described by the constant electric permittivity $\epsilon_1$ and the magnetic permeability $\mu_1 .$ The exterior domain $\Omega_{ext} : = \R^3 \setminus \overline{\Omega}_{int}$ is characterized equivalently by the constant coefficients $\epsilon_0$ and $\mu_0 .$ 
The smooth boundary $\partial \Omega$ of the cylinder consists of two disjoint surfaces $\partial \Omega_1$ and $\partial \Omega_0$ such that $\partial \Omega = \partial \Omega_1 \cup \partial \Omega_0,$  and $\partial \Omega_1 \cap \partial \Omega_0 = \emptyset .$ We assume that $\partial \Omega_1$ is contained in the interior of $\partial \Omega_0.$

We define the exterior electric and magnetic fields $\Ee , \He : \Omega_{ext} \rightarrow \C^3,$ respectively and the interior fields $\Ei , \Hi : \Omega_{int} \rightarrow \C^3,$ which satisfy the system of Maxwell's equations 
\begin{equation}\label{eq_Maxwell}
  \begin{aligned}
\curl \Ee - i \omega \mu_0  \He  &= 0, & \curl \He + i \omega \epsilon_0 \Ee &= 0,  & \mbox{in  } \Omega_{ext} ,\\
\curl \Ei - i \omega \mu_1  \Hi &= 0, & \curl \Hi + i \omega \epsilon_1  \Ei  &= 0,  
& \mbox{in  }  \Omega_{int} ,
\end{aligned}
\end{equation}
where $\omega >0$ is the frequency. We impose transmission conditions on the outer boundary
\begin{equation}\label{bound_cond_ext}
\n \times \Ei  =   \n \times \Ee, \quad \n \times \Hi  =   \n \times \He,  \quad \mbox{on  } \partial \Omega_0 ,
\end{equation}
 and the Leontovich impedance boundary condition on the inner boundary
\begin{equation}\label{bound_cond_int}
(\n \times \Ei ) \times \n =  \lambda \, \n \times \Hi  ,  \quad \mbox{on  } \partial \Omega_1 .
\end{equation}
Here $\n$ is the normal vector and  $\lambda \in C^1 (\partial \Omega_1)$ is the impedance function. These conditions model a penetrable cylinder which does not allow the fields to penetrate deep into the ``hole", the  simply-connected domain $\R^3 \setminus (\overline{\Omega}_{int} \cup \overline{\Omega}_{ext}).$

The scatterer is illuminated by a time-harmonic transverse magnetic polarized electromagnetic plane wave, the so-called oblique incident wave. The cylindrical symmetry and the homogeneity of the medium reduces the three-dimensional scattering problem \eqref{eq_Maxwell} -- \eqref{bound_cond_int} to a two-dimensional problem only for the $z-$components of the fields, see for instance \cite{GinMin16, NakWan13, WanNak12}.

\begin{figure}[t!]
\begin{center}
\includegraphics[scale=0.65]{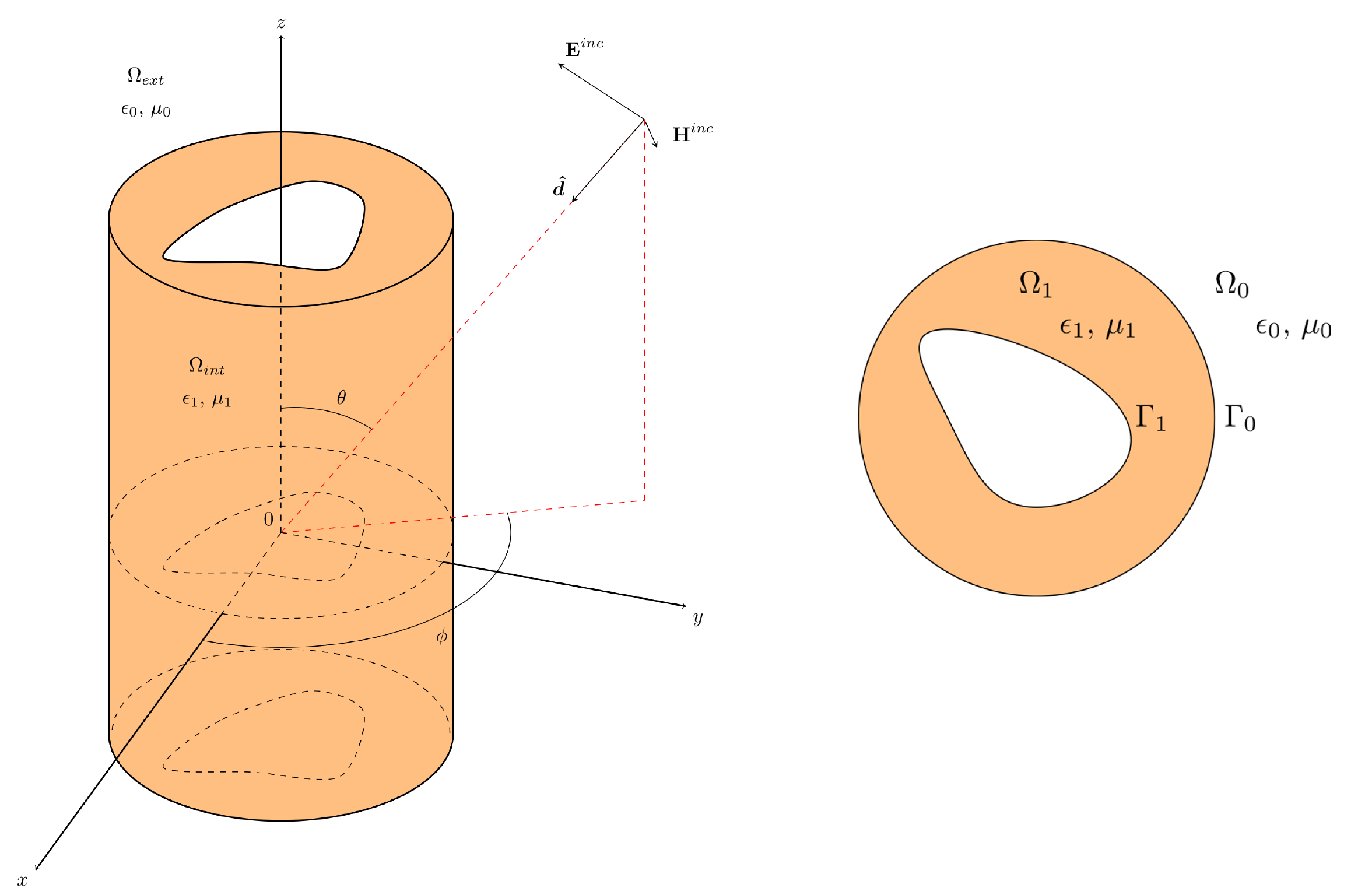}
\caption{The geometry of the electromagnetic scattering problem in $
\R^3$ for a doubly-connected cylinder oriented parallel to the $z-$axis (left). The horizontal cross section at the plane $z=0$  and the notation used in the two-dimensional problem (right).}
\label{fig1setup}
\end{center}
\end{figure}

We define by $\theta \in (0,\pi )$  the incident angle with respect to the negative $z-$axis and by $\phi \in [0,2\pi]$ the polar angle of the incident direction $\bm{\hat{d}}$, see the left picture in \autoref{fig1setup}. Let $k_0 = \omega \sqrt{\mu_0 \epsilon_0}$  be the wave number in $\Omega_{ext}$. We define $\beta = k_0 \cos \theta, $ $\kappa_0^2 = k_0^2 - \beta^2,$ and $\kappa^2_1 = \mu_1 \epsilon_1 \omega^2 - \beta^2 ,$ assuming that $\mu_1 \epsilon_1 >  \mu_0 \epsilon_0\cos^2\theta$ such that $\kappa^2_1 >0.$
We denote by $\Omega_1$ the horizontal cross section of the cylinder. Then, $\Omega_1$ is a doubly-connected bounded domain in $\R^2$ with a $C^2$ smooth boundary $\Gamma,$ consisting of two disjoint closed curves $\Gamma_1$ and $\Gamma_0$ such that $\Gamma = \Gamma_1 \cup \Gamma_0,$  and $\Gamma_1 \cap \Gamma_0 = \emptyset ,$ see the right picture in \autoref{fig1setup}.

Let $\b x = (x,y) \in\R^2.$ Then, the exterior fields (the $z-$components of $\Ee , \, \He$) defined by $\ee (\b x), \, \he (\b x),$ for $\b x \in\Omega_0$ and the interior fields $e^1 (\b x), \, h^1 (\b x),\,\b x \in\Omega_1 $ (the $z-$components of $\Ei , \, \Hi$), satisfy the system of Helmholtz equations
\begin{equation}\label{helm}
\begin{aligned}
\Delta \ee + \kappa^2_0 \,\ee &= 0,  &
\Delta \he + \kappa^2_0 \,\he &= 0,  & \mbox{in  } \Omega_0 , \\
\Delta e^1 + \kappa^2_1 \,e^1 &= 0, &
\Delta h^1 + \kappa^2_1 \,h^1 &= 0, & \mbox{in  } \Omega_1 .
\end{aligned}
\end{equation}

The boundary conditions \eqref{bound_cond_ext} and \eqref{bound_cond_int} can also be rewritten only for the $z-$components of the fields. Let $\b n = (n_1 , n_2)$ and $\bm \tau = (-n_2 , n_1)$ be the normal and tangent vector on $\Gamma ,$ respectively. The vector $\b n$ on $\Gamma_j$ points into $\Omega_j , \, j=0,1.$
We define $\tfrac{\partial}{\partial n } = \b n \cdot \nabla, \, \tfrac{\partial}{\partial \tau } = \bm \tau \cdot \nabla ,$ 
where $\nabla$ is the two-dimensional gradient and we set
 $$\tilde\mu_j = \frac{ \mu_j }{ \kappa_j^2} , \quad \tilde\epsilon_j = \frac{\epsilon_j }{ \kappa_j^2} , \quad \beta_j = \frac{\beta }{ \kappa_j^2 }, \quad \mbox{for   } j=0,1.$$

Then, the transmission conditions \eqref{bound_cond_ext} take the form \cite{GinMin16}
\begin{equation}\label{boundary_ext}
\begin{aligned}
e^1 &= \ee , & \mbox{on   } \Gamma_0 , \\
\tilde\mu_1 \omega \frac{\partial h^1}{\partial n }  + \beta_1 \frac{\partial e^1}{\partial \tau } &= \tilde\mu_0 \omega \frac{\partial \he}{\partial n }  + \beta_0 \frac{\partial \ee}{\partial \tau }, & \mbox{on   } \Gamma_0 , \\
h^1 &= \he , & \mbox{on   } \Gamma_0 , \\
\tilde\epsilon_1 \omega \frac{\partial e^1}{\partial n }  - \beta_1 \frac{\partial h^1}{\partial \tau } &= \tilde\epsilon_0 \omega \frac{\partial \ee}{\partial n }  - \beta_0 \frac{\partial \he}{\partial \tau }, & \mbox{on   } \Gamma_0 ,
\end{aligned}
\end{equation}
and the impedance boundary condition results to \cite{NakWan13}
\begin{subequations}\label{boundary_int}
\begin{alignat}{2}
 \tilde\mu_1 \omega   \frac{\partial h^1}{\partial n } + \beta_1 \frac{\partial e^1}{\partial \tau } +\lambda i h^1 &= 0 , & \quad \mbox{on   } \Gamma_1 , \label{interior1}\\
\lambda \tilde\epsilon_1 \omega \frac{\partial e^1}{\partial n }  - \lambda\beta_1 \frac{\partial h^1}{\partial \tau } + i e^1 &= 0, & \quad \mbox{on   } \Gamma_1 . \label{interior2}
\end{alignat}
\end{subequations}

The exterior fields are decomposed as $\ee = e^0 + e^{inc}$ and $\he = h^0 + h^{inc},$ where $e^0$ and $h^0$ is the scattered electric and magnetic field, respectively. The incident wave $(\b E^{inc}, \, \b H^{inc})$ reduces similarly to the fields ($z-$components) \cite{GinMin16, NakWan12b}
\begin{equation}\label{incident}
\begin{aligned}
e^{inc} (\b x) = \frac1{\sqrt{\epsilon_0}} \sin \theta \, \e^{i\kappa_0 (x \cos \phi +y \sin \phi )}, \quad
h^{inc} (\b x) = 0.
\end{aligned}
\end{equation}

The scattered fields satisfy also the radiation conditions
\begin{equation}\label{radiation}
\begin{aligned}
\lim_{r \rightarrow \infty} \sqrt{r} \left( \frac{\partial e^0}{\partial r} - i\kappa_0 e^0 \right) =0 , \quad
 \lim_{r \rightarrow \infty} \sqrt{r} \left( \frac{\partial h^0}{\partial r} - i\kappa_0 h^0 \right) =0 , \\
\end{aligned}
\end{equation}   
where $r = |\b x |,$ uniformly over all directions. 

The solutions $e^0$ and $h^0$ of  \eqref{helm} -- \eqref{radiation} admit the asymptotic behavior
\begin{equation}\label{far}
e^0  (\b x ) = \frac{\e^{i\kappa_0 r }}{\sqrt{r}} e^\infty (\b{\hat{x}}) + \mathcal{O} ( r^{-3/2}) , \quad h^0  (\b x ) = \frac{\e^{i\kappa_0 r }}{\sqrt{r}} h^\infty (\b{\hat{x}}) + \mathcal{O} (r^{-3/2}),
\end{equation}
where $\b{\hat{x}} = \b x / r.$ The pair $(e^\infty, h^\infty)$ is called the far-field pattern of the scattered fields related to the scattering problem  \eqref{helm} -- \eqref{radiation}.
We can formulate now the direct problem which we consider in this work.

\begin{description}
\item[Direct Problem:] Given the coefficients $\lambda, \,\epsilon_0 , \, \mu_0 , \epsilon_1 $ and $\mu_1$, the boundary curve $\Gamma = \Gamma_0 \cup \Gamma_1 ,$ and the incident field \eqref{incident}, find the interior fields $e^1$ and $h^1$ and the scattered fields $e^0$ and $h^0$ which satisfy the system of Helmholtz equations \eqref{helm}, the boundary conditions \eqref{boundary_ext} and \eqref{boundary_int} and the radiation conditions \eqref{radiation}.

\end{description}

\begin{remark}
Similar analysis holds also for the case of transverse electric polarized incident wave. The case of normal incidence $\theta = \pi /2,$ resulting to $\beta_1 = \beta_0 = 0,$ simplifies even more the scattering problem since it can be written as two decoupled problems for the electric and the magnetic field. 
\end{remark}

\section{Uniqueness results}\label{sec_unique}

In this section we study the well-posedness of the direct problem. We use the integral equation method and we apply the Reisz-Fredholm theory. For the representation of the electric and magnetic (interior and exterior) fields we consider a hybrid method, meaning we combine the direct (Green's formulas) and the indirect (single layer ansatz) methods \cite{KleMar88}. We define the ``hole" as $\Omega_h := \R^2 \setminus (\overline{\Omega}_1 \cup \overline{\Omega}_0).$

\begin{theorem}\label{theorem1}
If $\kappa_1^2$ is not a Dirichlet eigenvalue in $\Omega_1$ and the impedance parameter $\lambda$ is positive, then the direct scattering problem \eqref{helm} -- \eqref{radiation} admits at most one solution.
\end{theorem}

\begin{proof}
It is enough to show that the corresponding homogeneous problem, has only the trivial solution, meaning, $e^1 = h^1 = 0,$ in $\Omega_1$ and $e^0 = h^0 = 0,$ in $\Omega_0.$ We consider a disk $S_r$ with center at the origin, radius $r>0,$ and boundary $\Gamma_r,$ which contains $\Omega_1 .$  We set $\Omega_r = S_r \setminus (\overline{\Omega}_1 \cup \overline{\Omega}_h) .$

The transmission conditions on $\Gamma_0$ now read
\begin{subequations}\label{boundary_ext_homo}
\begin{alignat}{2}
e^1 &= e^0 , & \quad \mbox{on   } \Gamma_0 , \label{homo1}\\ 
\tilde\mu_1 \omega \frac{\partial h^1}{\partial n }  + \beta_1 \frac{\partial e^1}{\partial \tau } &= \tilde\mu_0 \omega \frac{\partial h^0}{\partial n }  + \beta_0 \frac{\partial e^0}{\partial \tau }, & \quad \mbox{on   } \Gamma_0 , \label{homo2}\\
h^1 &= h^0 , & \quad \mbox{on   } \Gamma_0 , \label{homo3}\\
\tilde\epsilon_1 \omega \frac{\partial e^1}{\partial n }  - \beta_1 \frac{\partial h^1}{\partial \tau } &= \tilde\epsilon_0 \omega \frac{\partial e^0}{\partial n }  - \beta_0 \frac{\partial h^0}{\partial \tau }, & \quad \mbox{on   } \Gamma_0 . \label{homo4}
\end{alignat}
\end{subequations}

We consider Green's first identity in $\Omega_1$ for the electric fields $e^1$ and $\overline{e^1},$ together with the Helmholtz equation \eqref{helm} and the boundary condition \eqref{interior2}, resulting in
\begin{equation}\label{green_omega1a}
\begin{aligned}
\tilde\epsilon_1 \int_{\Gamma_0} e^1 \frac{\partial \overline{e^1} }{\partial n } \, ds  &= \tilde\epsilon_1 \int_{\Omega_1}  \left( |\nabla e^1  |^2  -\kappa_1^2 | e^1 |^2 \right) d \b x + \int_{\Gamma_1} e^1 \left( \frac{\beta_1}{\omega}\frac{\partial \overline{ h^1}}{\partial \tau }  + \frac{i}{\lambda \omega} \overline{e^1}\right) ds .
\end{aligned}
\end{equation}
Similarly, Green's first identity in $\Omega_1$ for the magnetic fields $h^1$ and $\overline{h^1},$ considering the Helmholtz equation \eqref{helm} and the boundary condition \eqref{interior1}, gives
\begin{equation}\label{green_omega1b}
\begin{aligned}
\tilde\mu_1 \int_{\Gamma_0} h^1 \frac{\partial \overline{h^1} }{\partial n } \, ds  &= \tilde\mu_1 \int_{\Omega_1}  \left( |\nabla h^1  |^2  -\kappa_1^2 | h^1 |^2 \right) d \b x + \int_{\Gamma_1} h^1 \left( -\frac{\beta_1}{\omega}\frac{\partial \overline{ e^1}}{\partial \tau }  + \frac{i \lambda}{ \omega} \overline{h^1}\right) ds .
\end{aligned}
\end{equation}
Applying Green's first identity in $\Omega_r$ for the exterior fields and considering the transmission conditions \eqref{homo4} and \eqref{homo2}, we obtain
\begin{equation}\label{green_omegaRa}
\begin{aligned}
\tilde\epsilon_0 \int_{\Gamma_r} e^0 \frac{\partial \overline{e^0} }{\partial n } \, ds  &= \tilde\epsilon_0 \int_{\Omega_r}  \left( |\nabla e^0  |^2  -\kappa_0^2 | e^0 |^2 \right) d \b x 
+ \int_{\Gamma_0} e^0 \left(\tilde\epsilon_1 \frac{\partial \overline{ e^1}}{\partial n }- \frac{\beta_1}{\omega}\frac{\partial \overline{ h^1}}{\partial \tau }  + \frac{\beta_0}{\omega} \frac{\partial \overline{ h^0}}{\partial \tau }\right) ds ,
\end{aligned}
\end{equation}
and
\begin{equation}\label{green_omegaRb}
\begin{aligned}
\tilde\mu_0 \int_{\Gamma_r} h^0 \frac{\partial \overline{h^0} }{\partial n } \, ds  &= \tilde\mu_0 \int_{\Omega_r}  \left( |\nabla h^0  |^2  -\kappa_0^2 | h^0 |^2 \right) d \b x 
+ \int_{\Gamma_0} h^0 \left(\tilde\mu_1 \frac{\partial \overline{ h^1}}{\partial n }+ \frac{\beta_1}{\omega}\frac{\partial \overline{ e^1}}{\partial \tau }  - \frac{\beta_0}{\omega} \frac{\partial \overline{ e^0}}{\partial \tau }\right) ds .
\end{aligned}
\end{equation}

We take the imaginary part of \eqref{green_omegaRa}, and using \eqref{homo1} and \eqref{green_omega1a}, we have that
\begin{equation*}
\begin{aligned}
\Im \left(\tilde\epsilon_0 \int_{\Gamma_r} e^0 \frac{\partial \overline{e^0} }{\partial n } \, ds \right) &= \Im \left(\frac{\beta_1}{\omega} \int_{\Gamma_1}  e^1  \frac{\partial \overline{ h^1}}{\partial \tau } ds- \frac{\beta_1}{\omega}\int_{\Gamma_0} e^1\frac{\partial \overline{ h^1}}{\partial \tau } ds + \frac{\beta_0}{\omega} \int_{\Gamma_0}e^0\frac{\partial \overline{ h^0}}{\partial \tau } ds \right) \\
&\phantom{=}+ \int_{\Gamma_1} \frac{1}{\lambda \omega} |e^1  |^2 ds.
\end{aligned}
\end{equation*}
Analogously, the imaginary part of \eqref{green_omegaRb}, considering \eqref{homo3} and \eqref{green_omega1b}, takes the form
\begin{equation*}
\begin{aligned}
\Im \left(\tilde\mu_0 \int_{\Gamma_r} h^0 \frac{\partial \overline{h^0} }{\partial n } \, ds \right) &= \Im \left(-\frac{\beta_1}{\omega} \int_{\Gamma_1}  h^1  \frac{\partial \overline{ e^1}}{\partial \tau } ds + \frac{\beta_1}{\omega}\int_{\Gamma_0} h^1\frac{\partial \overline{ e^1}}{\partial \tau }ds  - \frac{\beta_0}{\omega} \int_{\Gamma_0}h^0\frac{\partial \overline{ e^0}}{\partial \tau } ds \right) \\
&\phantom{=}+ \int_{\Gamma_1}\frac{\lambda}{ \omega} |h^1  |^2 ds.
\end{aligned}
\end{equation*}

If $\lambda>0,$ the addition of the above two equations, noting that
\[
-\int_{\Gamma_j} e^k \frac{\partial \overline{ h^k}}{\partial \tau } ds = \overline{\int_{\Gamma_j} h^k \frac{\partial \overline{ e^k}}{\partial \tau } ds}, \quad \mbox{for } k,j= 0,1,
\]
results to
\begin{equation*}
\Im \left(\tilde\epsilon_0 \int_{\Gamma_r} e^0 \frac{\partial \overline{e^0} }{\partial n } \, ds + \tilde\mu_0 \int_{\Gamma_r} h^0 \frac{\partial \overline{h^0} }{\partial n } \, ds\right) = \int_{\Gamma_1} \left( \frac{1}{\lambda \omega} |e^1  |^2 +\frac{\lambda}{ \omega}  |h^1  |^2 \right) ds \geq 0
\end{equation*}

The last equation, the radiation conditions \eqref{radiation} as $r\rightarrow \infty$ and Rellich's Lemma yield $e^0 = h^0 = 0$ in $\Omega_0$ \cite{GinMin16, WanNak12}. Hence, $e^0 = h^0 = 0$ on $\Gamma_0.$ Using the homogeneous transmission conditions \eqref{homo1} and \eqref{homo3} and the assumption on $\kappa_1^2$ we get also $e^1 = h^1 = 0$ in $\Omega_1.$ This completes the proof.
\end{proof}

To prove existence of solutions we transform the direct problem to a system of boundary integral equations. We present the fundamental solution of the Helmholtz equation in $\R^2,$ given by
\begin{equation}
\Phi_j (\b x,\b y) = \frac{i}4 H_0^{(1)} (\kappa_j |\b x-\b y|), \quad \b x,\b y \in\Omega_j , \quad \b x \neq \b y,  
\end{equation}
where $H_0^{(1)}$ is the Hankel function of the first kind and zero order. We introduce the single- and double-layer potentials for a continuous density $f,$ given by
\begin{equation*}
\begin{aligned}
(\mathcal S_{klj} f) (\b x) &= \int_{\Gamma_j} \Phi_k (\b x,\b y) f(\b y) ds (\b y), & \b x \in\Omega_l , \\
(\mathcal D_{klj} f) (\b x) &= \int_{\Gamma_j} \frac{\partial \Phi_k}{\partial n (\b y)} (\b x,\b y) f(\b y) ds (\b y), & \b x \in\Omega_l ,
\end{aligned}
\end{equation*}
for $k,l,j = 0,1.$ The single-layer potential $\mathcal S$ is continuous in $\R^2$ and the their normal and tangential derivatives as $\b x \rightarrow \Gamma_j$ satisfy the standard jump relations, see for instance \cite{GinMin16}. We define the integral operators
\begin{equation}\label{operators}
\begin{aligned}
(S_{klj} f) (\b x) &= \int_{\Gamma_j} \Phi_k (\b x,\b y) f(\b y) ds (\b y), & \b x \in\Gamma_l ,\\
( D_{klj} f) (\b x) &= \int_{\Gamma_j} \frac{\partial \Phi_k}{\partial n (\b y)} (\b x,\b y) f(\b y) ds (\b y), & \b x \in\Gamma_l , \\
( NS_{klj} f) (\b x) &= \int_{\Gamma_j} \frac{\partial \Phi_k}{\partial n (\b x)} (\b x,\b y) f(\b y) ds (\b y), & \b x \in\Gamma_l , \\
( ND_{klj} f) (\b x) &= \int_{\Gamma_j} \frac{\partial^2 \Phi_k}{\partial n (\b x)\partial n (\b y)} (\b x,\b y) f(\b y) ds (\b y), & \b x \in\Gamma_l , \\
( TS_{klj} f) (\b x) &= \int_{\Gamma_j} \frac{\partial \Phi_k}{\partial \tau (\b x)} (\b x,\b y) f(\b y) ds (\b y), & \b x \in\Gamma_l , \\
( TD_{klj} f) (\b x) &= \int_{\Gamma_j} \frac{\partial^2 \Phi_k}{\partial \tau (\b x)\partial n (\b y)} (\b x,\b y) f(\b y) ds (\b y), & \b x \in\Gamma_l .
\end{aligned}
\end{equation}

If we consider the direct method, meaning Green's second identity, for representing the interior and exterior electric and magnetic fields we get
\begin{equation*}
\begin{aligned}
u^1 (\b x) &= (\mathcal S_{110} \, \partial_n u^1) (\b x) - (\mathcal D_{110} \, u^1) (\b x) + (\mathcal S_{111} \, \partial_n u^1) (\b x) + (\mathcal D_{111} \, u^1) (\b x), & \quad \b x \in \Omega_1 , \\
u^0 (\b x) &= (\mathcal D_{000} \, u^0) (\b x) - (\mathcal S_{000} \, \partial_n u^0) (\b x) , & \quad \b x \in \Omega_0 ,
\end{aligned}
\end{equation*}
for $u = e,h.$ We observe that we have 8 unknown densities (4 for the electric and 4 for the magnetic field) for the interior fields and 4 unknown densities for the exterior fields and only 6 equations (the boundary conditions \eqref{boundary_ext} and \eqref{boundary_int}). In order to reduce the number of the unknowns, motivated by \cite{KleMar88}, we consider a hybrid method, meaning a combination of the indirect and direct methods. We keep the direct method for the exterior fields and we consider a single-layer ansatz (indirect method) for the interior fields. 

\begin{theorem}\label{theorem2}
Let the assumptions of \autoref{theorem1} still hold. If $\kappa_1^2$ is not a Dirichlet eigenvalue in $\Omega_h,$ and  $\kappa_0^2$ is not a Dirichlet eigenvalue in $\R^2 \setminus \overline{\Omega}_0,$ then the problem \eqref{helm} -- \eqref{radiation} has a unique solution.
\end{theorem}

\begin{proof}

We search the solutions in the forms:
\begin{equation}\label{fields}
\begin{aligned}
e^1 (\b x) &= (\mathcal S_{110} \, \psi_1^e ) (\b x)  + (\mathcal S_{111} \, \psi_2^e) (\b x), & \quad \b x \in \Omega_1 , \\
h^1 (\b x) &= (\mathcal S_{110} \, \psi_1^h ) (\b x)  + (\mathcal S_{111} \, \psi_2^h) (\b x), & \quad \b x \in \Omega_1 , \\
e^0 (\b x) &= (\mathcal D_{000} \, \phi_0^e) (\b x) -(\mathcal S_{000} \, \psi_0^e ) (\b x) , & \quad \b x \in \Omega_0 , \\
h^0 (\b x) &=  (\mathcal D_{000} \, \phi_0^h) (\b x)-  (\mathcal S_{000} \, \psi_0^h ) (\b x) , & \quad \b x \in \Omega_0 ,
\end{aligned}
\end{equation}
with $\psi_0^u := \partial_n u^0 |_{\Gamma_0}$ and $\phi_0^u := u^0 |_{\Gamma_0},$ for $u=e,h.$ We let $\b x$ approach the boundaries $\Gamma_j,$ and considering the jump relations of the potentials \cite{ColKre13b, CouHil62}, we see that the boundary conditions \eqref{boundary_ext} and \eqref{boundary_int} are satisfied if the densities $\psi_k^u, \, \phi_0^u, \, k=0,1,2, \, u =e,h$ solve the system of integral equations \small
\begin{equation}\label{system1}
\begin{aligned}
S_{100} \psi_1^e + S_{101} \psi_2^e - \left(D_{000} +\tfrac12 \right)\phi_0^e + S_{000}\psi_0^e &= e^{inc}, \\
\tilde{\mu}_1 \omega \left(NS_{100} +\tfrac12 \right) \psi_1^h +\tilde{\mu}_1 \omega NS_{101} \psi_2^h +\beta_1 TS_{100} \psi_1^e +\beta_1 TS_{101} \psi_2^e - \tilde \mu_0 \omega ND_{000} \phi_0^h \\+\tilde \mu_0 \omega \left(NS_{000} -\tfrac12 \right) \psi_0^h -\beta_0 (TD_{000} +\tfrac{\partial_\tau}{2}) \phi_0^e + \beta_0 TS_{000} \psi_0^e &=\beta_0 \partial_\tau e^{inc}, \\
S_{100} \psi_1^h + S_{101} \psi_2^h - \left(D_{000} +\tfrac12 \right)\phi_0^h + S_{000}\psi_0^h &=0 \\
\tilde{\epsilon}_1 \omega \left(NS_{100} +\tfrac12 \right) \psi_1^e +\tilde{\epsilon}_1 \omega NS_{101} \psi_2^e -\beta_1 TS_{100} \psi_1^h -\beta_1 TS_{101} \psi_2^h - \tilde \epsilon_0 \omega ND_{000} \phi_0^e \\+\tilde \epsilon_0 \omega \left(NS_{000} -\tfrac12 \right) \psi_0^e +\beta_0 (TD_{000} +\tfrac{\partial_\tau}{2}) \phi_0^h - \beta_0 TS_{000} \psi_0^h &=\tilde \epsilon_0 \omega \partial_n e^{inc}, \\
\tilde{\mu}_1 \omega NS_{110}  \psi_1^h +\tilde{\mu}_1 \omega \left(NS_{111}-\tfrac12 \right) \psi_2^h +\beta_1 TS_{110} \psi_1^e +\beta_1 TS_{111} \psi_2^e +i \lambda S_{110} \psi_1^h + i \lambda S_{111} \psi_2^h &=0, \\
\lambda \tilde{\epsilon}_1 \omega NS_{110}  \psi_1^e +\lambda\tilde{\epsilon}_1 \omega \left(NS_{111}-\tfrac12 \right) \psi_2^e -\lambda \beta_1 TS_{110} \psi_1^h -\lambda\beta_1 TS_{111} \psi_2^h + i S_{110} \psi_1^e + i  S_{111} \psi_2^e &=0 .
\end{aligned}
\end{equation}

To simplify the above system, we set
\begin{equation}\label{assume}
\tilde \epsilon_1 \psi_1^e = - \tilde \epsilon_0 \psi_0^e , \quad \mbox{and} \quad \tilde \mu_1 \psi_1^h = - \tilde \mu_0 \psi_0^h ,
\end{equation}
and the system \eqref{system1} admits the matrix form
\begin{equation}\label{system2}
\left(\b B + \b C \right) \bm\phi = \b f,
\end{equation}
where $\bm\phi = ( \phi_0^e , \, \psi_1^h, \, \phi_0^h , \, \psi_1^e, \, \psi_2^h, \, \psi_2^e )^\top \in \C^6, \quad \b f = \left( e^{inc}, \,\beta_0 \partial_\tau e^{inc}, \, 0, \, \tilde \epsilon_0 \omega \partial_n e^{inc}, \, 0, \, 0\right)^\top \in \C^6,$ 
and
\begin{equation*}
\b B = \begin{pmatrix}
  -\dfrac12 & 0 & 0 & 0 & 0 & 0\\[10pt]
  -\dfrac{\beta_0}{2}\partial_\tau & \tilde\mu_1 \omega & 0 & 0 & 0 & 0\\
  0 & 0 & -\dfrac12 & 0 & 0 & 0 \\[10pt]
  0 & 0 & \dfrac{\beta_0}{2}\partial_\tau & \tilde\epsilon_1 \omega & 0 & 0 \\
  0 & 0 & 0 & 0 & -\dfrac{\tilde \mu_1 \omega}{2} & 0\\
  0 & 0 & 0 & 0 & 0 & -\dfrac{\lambda\tilde \epsilon_1 \omega}{2}   
\end{pmatrix} .
\end{equation*}
The operator $\b C = (C_{kj})_{1\leq k,j \leq 6}$ has entries:
\begin{align*}
C_{11} &= -D_{000}  , &
C_{14} &= S_{100} - \dfrac{\tilde{\epsilon}_1}{\tilde{\epsilon}_0} S_{000}, & 
C_{16} &= S_{101}, \\
C_{21} &= -\beta_0 TD_{000}, & C_{22} &=  \tilde{\mu}_1 \omega (NS_{100}-NS_{000}), & C_{23} &= -\tilde{\mu}_0 \omega ND_{000}, \\
C_{24} &= \beta_1 TS_{100} - \beta_0\dfrac{ \tilde{\epsilon}_1}{\tilde{\epsilon}_0} TS_{000}, & C_{25} &= \tilde{\mu}_1 \omega NS_{101}, & C_{26} &= \beta_1 TS_{101}, \\
C_{32} &= S_{100} - \dfrac{\tilde{\mu}_1}{\tilde{\mu}_0} S_{000}, & C_{33} &= C_{11}, & C_{35} &= C_{16}, \\
C_{41} &= -\tilde{\epsilon}_0 \omega ND_{000}, & C_{42} &= -\beta_1 TS_{100} + \beta_0\dfrac{ \tilde{\mu}_1}{\tilde{\mu}_0} TS_{000}, & C_{43} &= -C_{21}, \\ 
C_{44} &= \tilde{\epsilon}_1 \omega (NS_{100}-NS_{000}), & C_{45} &= -C_{26}, & C_{46} &= \tilde{\epsilon}_1 \omega NS_{101}, \\ 
C_{52} &= \tilde{\mu}_1 \omega NS_{110} +i \lambda S_{110}, & C_{54} &= \beta_1 TS_{110}, & C_{55} &= \tilde{\mu}_1 \omega NS_{111} +i \lambda S_{111}, \\ 
C_{56} &= \beta_1 TS_{111}, & C_{62} &= -\lambda C_{54}, & C_{64} &= \lambda\tilde{\epsilon}_1 \omega NS_{110} +i S_{110}, \\ 
C_{65} &= -\lambda C_{56}, & C_{66} &= \lambda\tilde{\epsilon}_1 \omega NS_{111} +i  S_{111},  \\ 
\end{align*}
and the rest are zero. The special form of $\b B$ and the boundedness of the tangential operator $\partial_\tau : H^{1/2} (\Gamma_0) \rightarrow H^{-1/2} (\Gamma_0)$ allow us to construct its bounded inverse, given by
\begin{equation*}
\b B^{-1} = \begin{pmatrix}
  -2 & 0 & 0 & 0 & 0 & 0\\[10pt]
  -\dfrac{\beta_0}{\tilde{\mu}_1 \omega}\partial_\tau & \dfrac1{\tilde\mu_1 \omega} & 0 & 0 & 0 & 0\\
  0 & 0 & -2 & 0 & 0 & 0 \\[10pt]
  0 & 0 & \dfrac{\beta_0}{\tilde{\epsilon}_1 \omega}\partial_\tau & \dfrac1{\tilde\epsilon_1 \omega} & 0 & 0 \\
  0 & 0 & 0 & 0 & -\dfrac{2}{\tilde \mu_1 \omega} & 0\\
  0 & 0 & 0 & 0 & 0 & -\dfrac{2}{\lambda\tilde \epsilon_1 \omega}   
\end{pmatrix} .
\end{equation*}
Then, we rewrite \eqref{system2} as
\begin{equation}\label{system3}
\left(\b I + \b K \right) \bm\phi = \b g,
\end{equation}
where $\b I$ is the identity operator, $\b g = \b B^{-1} \b f = \left( -2 e^{inc}, \,0, \, 0, \, \dfrac{\tilde \epsilon_0}{\tilde \epsilon_1} \partial_n e^{inc}, \, 0, \, 0\right)^\top ,$ and the matrix $\b K = \b B^{-1} \b C$ has now entries:
\begin{equation*}
K_{1j} = -2 C_{1j}, \quad K_{3j} = -2 C_{3j}, \quad K_{5j} = -\dfrac2{\tilde{\mu}_1 \omega} C_{5j}, \quad K_{6j} = -\dfrac{2}{\lambda\tilde \epsilon_1 \omega} C_{6j}, \quad \mbox{for } j=1,...,6,
\end{equation*}
and
\begin{equation*}
\begin{aligned}
K_{21} &= 0, & K_{22} &= NS_{100}-NS_{000}, & K_{23} &= -\dfrac{\tilde{\mu}_0}{\tilde{\mu}_1} ND_{000}, & K_{24} &= \dfrac{\beta_1 - \beta_0}{\tilde{\mu}_1 \omega} TS_{100}, \\
K_{25} &= NS_{101}, & K_{26} &= \dfrac{\beta_1 - \beta_0}{\tilde{\mu}_1 \omega} TS_{101}, & K_{41} &= -\dfrac{\tilde{\epsilon}_0}{\tilde{\epsilon}_1} ND_{000}, & K_{42} &= \dfrac{\beta_0 - \beta_1}{\tilde{\epsilon}_1 \omega} TS_{100}, \\
 K_{43} &=0, & K_{44} &= K_{22}, & K_{45} &= \dfrac{\beta_0 - \beta_1}{\tilde{\epsilon}_1 \omega} TS_{101}, & K_{46} &= K_{25} .
\end{aligned}
\end{equation*}

We define the product spaces:
\begin{align*}
H_1 &:= \left( H^{1/2} (\Gamma_0) \times H^{-1/2} (\Gamma_0)\right)^2 \times \left(  H^{-1/2} (\Gamma_1)\right)^2 , \\
H_2 &:= \left( H^{-1/2} (\Gamma_0) \times H^{-3/2} (\Gamma_0)\right)^2 \times \left(  H^{-3/2} (\Gamma_1)\right)^2 ,
\end{align*}
and using the mapping properties of the integral operators \cite{ColKre13b, Kre99}, we see that the operator $\b K : H_1 \rightarrow H_2$ is compact. The last step is to prove uniqueness of solutions. Then, solvability of the system \eqref{system3} follows from the Fredholm alternative theorem. It is sufficient to show that the operator $\b I +\b K$ is injective. 

Let $\bm\phi$ solve $\left(\b I + \b K \right) \bm\phi = \b 0.$ Then, the fields \eqref{fields} solve the problem \eqref{helm} -- \eqref{radiation} with  $e^{inc} = \partial_n  e^{inc}= \partial_\tau e^{inc} = 0,$ on $\Gamma_0.$ Hence, by \autoref{theorem1} we have $e^1 = h^1 =0,$ in $\Omega_1$ and $e^0 = h^0 =0,$ in $\Omega_0.$ Continuity  of the single-layer potential implies that $e^1$ and $h^1$ solve also 
\[
\Delta e^1 + \kappa^2_1 \,e^1 = 0, \quad
\Delta h^1 + \kappa^2_1 \,h^1 = 0, \quad \mbox{in  } \Omega_h ,
\]
and vanish on $\Gamma_1.$ Thus, $e^1 = h^1 =0,$ in $\Omega_h,$ if $\kappa_1^2$ is not a Dirichlet eigenvalue in $\Omega_h .$ The jump relation of the normal derivative of the single layer across the boundary $\Gamma_1,$ gives $\psi_2^e = \psi_2^h =0,$ on $\Gamma_1 .$

We define 
\[
\tilde e (\b x) = (\mathcal S_{100} \, \psi_1^e ) (\b x), \quad \tilde h (\b x) = (\mathcal S_{100} \, \psi_1^h ) (\b x), \quad \b x \in \Omega_0 .
\]
Again using the continuity of the single-layer potential we get $\tilde{e} = S_{100} \psi_1^e = e^1 = 0,$ and $\tilde{h} = S_{100} \psi_1^h = h^1 = 0,$ on $\Gamma_0.$ Since $\tilde{e}$ and $\tilde{h}$ are also radiating solutions of the Helmholtz equation in $\Omega_0,$ we get $\tilde{e} = \tilde{h} = 0,$ in $\Omega_0$ \cite{ColKre13b}. Again the jump relation of the normal derivative of the single layer potential across the boundary $\Gamma_0$, results to $\psi_1^e = \psi_1^h = 0,$ on $\Gamma_0.$ From \eqref{assume} we have also $\psi_0^e = \psi_0^h = 0,$ on $\Gamma_0.$ 

Given now the representations  \eqref{fields}, the homogeneous transmission conditions \eqref{homo1} and \eqref{homo3} and the jump relations of the double-layer potential we obtain the equations
\[
\left( D_{000} - \frac12\right) \phi_0^e =0, \quad \left( D_{000} - \frac12\right) \phi_0^h =0, \quad \mbox{on } \Gamma_0 .
\]
This integral operator is injective if $\kappa_0^2$ is not a Dirichlet eigenvalue in $\R^2 \setminus \overline{\Omega}_0.$ Thus, $\phi_0^e =\phi_0^h =0,$ on $\Gamma_0,$ and therefore $\bm\phi = 0$ which completes the proof.
\end{proof}

\section{Numerical examples}\label{sec_numerics}

We derive the numerical solution of \eqref{system3} by a collocation method using trigonometric polynomial approximations \cite{Kre99}. We use quadrature rules to handle the singularities (weak and strong) of the integral operators and the trapezoidal rule for approximating the smooth kernels \cite{Kre14}. For the convergence and error analysis, see \cite{Kre95}.
We do not present here the parametrized forms of the integral operators  \eqref{operators} since they can be found in many works, we refer the reader to the books \cite{ColKre13b, Kre99} and to \cite[Section 4]{GinMin16} for a complete list of all forms and special decompositions.

We present two different kind of problems. In the first case, we formulate a problem, as in \cite{GinMin16, WanNak12}, whose solutions (the scattered electric and magnetic fields) can be analytically calculated. In this way, we can see that the exponential convergence is achieved. Secondly, we present results for the scattering problem of obliquely incident waves. 

We assume that the smooth boundaries admit the following parametrization
\[
\Gamma_j = \{ \b x^j (t) = (x^j_1 (t), x^j_2 (t)) : t \in [0,2\pi]\}, \quad j=0,1,
\]
where $\b  x^j : \R \rightarrow \R^2$ are $C^2$-smooth, $2\pi$-periodic and injective in $[0,2\pi).$ For the numerical implementation we consider equidistant grid points
\[
t_k = \frac{k\pi}{n}, \quad k=0,...,2n-1.
\]

\begin{figure}[t]
\begin{center}
\includegraphics[scale=0.85]{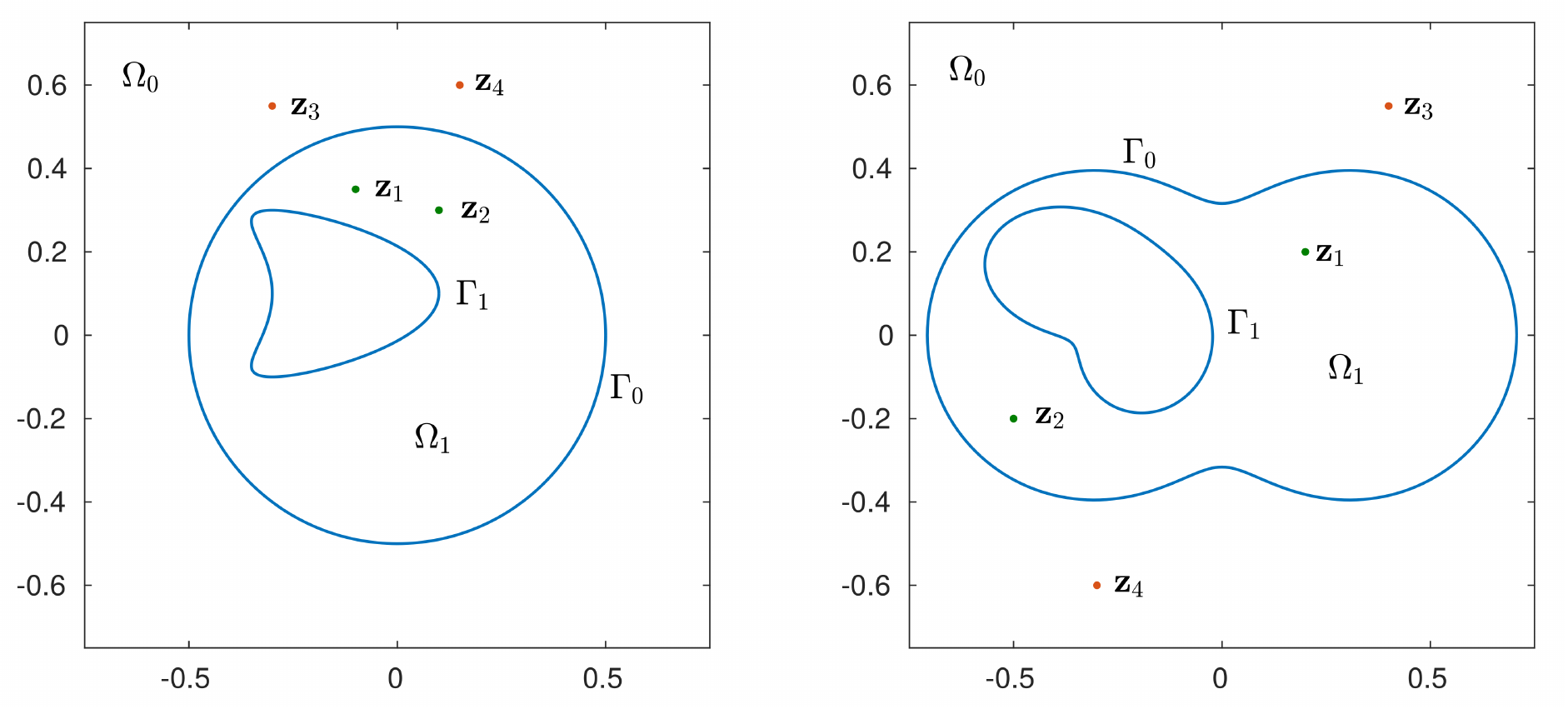}
\caption{The parametrization of the boundary $\Gamma = \Gamma_1 \cup \Gamma_0$ and the source points in the first (left) and in the second example (right).}\label{Fig2}
\end{center}
\end{figure}

\subsection{Examples with analytic solution} 

We construct a problem whose solutions are expressed analytically. Let $\b z_1 , \, \b z_2 \in \Omega_1$ and $\b z_3 , \, \b z_4 \in \Omega_0$ be four arbitrary points. We define the boundary functions
\begin{align*}
f_1 (\b x) &= H_0^{(1)} (\kappa_1 |\b r_3 (\b x) |) - H_0^{(1)} (\kappa_0 |\b r_1 (\b x) |) , & \b x \in \Gamma_0 ,\\
f_2 (\b x)&= -\tilde\mu_1 \omega \kappa_1 H_1^{(1)} (\kappa_1 |\b r_4 (\b x) |) \frac{\n (\b x) \cdot \b r_4 (\b x)}{|\b r_4 (\b x) |} 
 - \beta_1 \kappa_1 H_1^{(1)} (\kappa_1 |\b r_3 (\b x) |) \frac{\ta (\b x) \cdot \b r_3 (\b x)}{|\b r_3 (\b x) |}\\
&\phantom{=}+\tilde\mu_0\omega \kappa_0 H_1^{(1)} (\kappa_0 |\b r_2 (\b x) |) \frac{\n (\b x) \cdot \b r_2 (\b x)}{|\b r_2 (\b x) |} +\beta_0 \kappa_0 H_1^{(1)} (\kappa_0 |\b r_1 (\b x) |) \frac{\ta (\b x) \cdot \b r_1 (\b x)}{|\b r_1 (\b x) |} , & \b x \in \Gamma_0 ,\\
f_3 (\b x)&= H_0^{(1)} (\kappa_1 |\b r_4 (\b x) |) -  H_0^{(1)} (\kappa_0 |\b r_2 (\b x) |) ,  & \b x \in \Gamma_0 ,\\
f_4 (\b x) &= -\tilde\epsilon_1 \omega \kappa_1 H_1^{(1)} (\kappa_1 |\b r_3 (\b x) |) \frac{\n (\b x) \cdot \b r_3 (\b x)}{|\b r_3 (\b x) |} + \beta_1 \kappa_1 H_1^{(1)} (\kappa_1 |\b r_4 (\b x) |) \frac{\ta (\b x) \cdot \b r_4 (\b x)}{|\b r_4 (\b x) |}\\
&\phantom{=}+\tilde\epsilon_0\omega \kappa_0 H_1^{(1)} (\kappa_0 |\b r_1 (\b x) |) \frac{\n (\b x) \cdot \b r_1 (\b x)}{|\b r_1 (\b x) |} -\beta_0 \kappa_0 H_1^{(1)} (\kappa_0 |\b r_2 (\b x) |) \frac{\ta (\b x) \cdot \b r_2 (\b x)}{|\b r_2 (\b x) |} ,  & \b x \in \Gamma_0 , \\
f_5 (\b x)&= -\tilde\mu_1 \omega \kappa_1 H_1^{(1)} (\kappa_1 |\b r_4 (\b x) |) \frac{\n (\b x) \cdot \b r_4 (\b x)}{|\b r_4 (\b x) |} 
 - \beta_1 \kappa_1 H_1^{(1)} (\kappa_1 |\b r_3 (\b x) |) \frac{\ta (\b x) \cdot \b r_3 (\b x)}{|\b r_3 (\b x) |}\\
&\phantom{=}+i \lambda  H_0^{(1)} (\kappa_1 |\b r_4 (\b x) |), & \b x \in \Gamma_1 ,\\
f_6 (\b x) &= -\lambda\tilde\epsilon_1 \omega \kappa_1 H_1^{(1)} (\kappa_1 |\b r_3 (\b x) |) \frac{\n (\b x) \cdot \b r_3 (\b x)}{|\b r_3 (\b x) |} + \lambda \beta_1 \kappa_1 H_1^{(1)} (\kappa_1 |\b r_4 (\b x) |) \frac{\ta (\b x) \cdot \b r_4 (\b x)}{|\b r_4 (\b x) |}\\
&\phantom{=}+i   H_0^{(1)} (\kappa_1 |\b r_3 (\b x) |), & \b x \in \Gamma_1 ,\\
\end{align*}
where $\b r_k (\b x) = \b x - \b z_k , k=1,2,3,4.$ Then, the fields
\begin{equation}\label{eq_examaple}
\begin{aligned}
e^0 (\b x) &= H_0^{(1)} (\kappa_0 |\b x- \b z_1 |), &  h^0 (\b x) &= H_0^{(1)} (\kappa_0 |\b x- \b z_2 |), \quad \b  x \in \Omega_0, \\
e^1 (\b x) &= H_0^{(1)} (\kappa_1 |\b x- \b z_3 |), &  h^1 (\b x) &= H_0^{(1)} (\kappa_1 |\b x- \b z_4 |), \quad \b  x \in \Omega_1 ,
\end{aligned}
\end{equation}
solve the equations
\begin{equation*}
\begin{aligned}
\Delta e^0 + \kappa^2_0 \,e^0 &= 0,  &
\Delta h^0 + \kappa^2_0 \,h^0 &= 0,  & \mbox{in  } \Omega_0 , \\
\Delta e^1 + \kappa^2_1 \,e^1 &= 0, &
\Delta h^1 + \kappa^2_1 \,h^1 &= 0, & \mbox{in  } \Omega_1 ,
\end{aligned}
\end{equation*}
with boundary conditions
\begin{equation*}
\begin{aligned}
e^1 - e^0 &= f_1 , & \mbox{on   } \Gamma_0 , \\
\tilde\mu_1 \omega \frac{\partial h^1}{\partial n }  + \beta_1 \frac{\partial e^1}{\partial \tau } - \tilde\mu_0 \omega \frac{\partial h^0}{\partial n }  - \beta_0 \frac{\partial e^0}{\partial \tau } &= f_2, & \mbox{on   } \Gamma_0 , \\
h^1 - h^0 &= f_3 , & \mbox{on   } \Gamma_0 , \\
\tilde\epsilon_1 \omega \frac{\partial e^1}{\partial n }  - \beta_1 \frac{\partial h^1}{\partial \tau } - \tilde\epsilon_0 \omega \frac{\partial e^0}{\partial n }  + \beta_0 \frac{\partial h^0}{\partial \tau } &= f_4, & \mbox{on   } \Gamma_0 , \\
\tilde\mu_1 \omega   \frac{\partial h^1}{\partial n } + \beta_1 \frac{\partial e^1}{\partial \tau } +\lambda i h^1 &= f_5 , & \quad \mbox{on   } \Gamma_1 , \\
\lambda \tilde\epsilon_1 \omega \frac{\partial e^1}{\partial n }  - \lambda\beta_1 \frac{\partial h^1}{\partial \tau } + i e^1 &= f_6, & \quad \mbox{on   } \Gamma_1 .
\end{aligned}
\end{equation*}
and $e^0 , \, h^0$ satisfy in addition the radiation conditions \eqref{radiation}.

 \begin{table}[t]
\begin{center}
 \begin{tabular}{| c  | c  | c  | } 
 \hline
 $n$ & $ e_n^\infty (\b{\hat{x}}(0)) $ & $ h_n^\infty (\b{\hat{x}}(0)) $  
\\ \hline 
8 & $ 0.550084263052      - i \, 0.665877312380   $ & $  0.646238687778     - i \, 0.549065505420  $  \\ 
16 & $ 0.550961953612      - i \, 0.656336124631  $ & $  0.656591070845     - i \, 0.551308759662   $  \\
32 & $ 0.551551006183      - i \, 0.656427458763  $ & $ 0.656427073431     - i \, 0.551550848571  $ \\
64 & $ 0.551550951843      - i \, 0.656427255249   $ & $ 0.656427255242     -i \, 0.551550951840 $ \\
\cline{1-1}\hhline{~==}
\multicolumn{1}{c|}{} &  $ e^\infty (\b{\hat{x}}(0)) $   & $h^\infty (\b{\hat{x}}(0))$ \\
\cline{2-3} 
\multicolumn{1}{c|}{} &  $ 0.551550951838      - i \, 0.656427255240  $  & $  0.656427255240     - i \, 0.551550951838   $  \\ \cline{2-3}
\end{tabular}
\caption{The computed and the exact far-fields of the electric and magnetic scattered fields of the first example. }\label{table1}
\end{center}
\end{table}

As in the proof of \autoref{theorem2} we derive a system of the form \eqref{system3}, where now the right-hand side is given by
\[
\b g_f = \left(
-2f_1 ,\,
-\dfrac{\beta_0}{\tilde{\mu}_1 \omega} \partial_\tau f_1 + \dfrac{1}{\tilde{\mu}_1 \omega} f_2 ,\,
-2 f_3 ,\,
\dfrac{\beta_0}{\tilde{\epsilon}_1 \omega} \partial_\tau f_3 + \dfrac{1}{\tilde{\epsilon}_1 \omega} f_4 ,\,
-\dfrac{2}{\tilde{\mu}_1 \omega} f_5 ,\,
-\dfrac{2}{\lambda \tilde{\epsilon}_1 \omega} f_6
\right)^\top.
\]

Using the asymptotic behavior of the Hankel functions \cite{ColKre13b}, we can construct the exact and the computed far-field pattern of the scattered fields. We see from \eqref{eq_examaple} that the exact values of the far-field patterns of $e^0$ and $h^0$ are given by
\begin{equation}\label{far_exact}
e^\infty (\b{\hat{x}}) = \frac{-4 i \e^{i\pi /4}}{\sqrt{8\pi \kappa_0}} \e^{-i\kappa_0 \b{\hat{x}} \cdot \b z_1}, \quad
h^\infty (\b{\hat{x}}) = \frac{-4 i \e^{i\pi /4}}{\sqrt{8\pi \kappa_0}} e^{-i\kappa_0 \b{\hat{x}} \cdot \b z_2}, \quad \b{\hat{x}} \in S,
\end{equation}
where $S$ is the unit circle. The representations \eqref{fields}, where now the densities solve \eqref{system3} with $\b g$ replaced by $\b g_f ,$   result in
\begin{equation}\label{far_comp}
\begin{aligned}
e_n^\infty (\b{\hat{x}}(t)) &= \frac{\e^{i\pi /4}}{\sqrt{8\pi \kappa_0}} \int_0^{2\pi} \e^{-i\kappa_0 \b{\hat{x}} \cdot \b x^0 (t)} \left[ -i\kappa_0  (\b{\hat{x}} \cdot \n (\b x^{0} (t))) \phi^e_0 (t)   - \psi_0^e (t) \right] |\b x^{0'} (t)| dt, \\
h_n^\infty (\b{\hat{x}}(t)) &= \frac{\e^{i\pi /4}}{\sqrt{8\pi \kappa_0}} \int_0^{2\pi} \e^{-i\kappa_0 \b{\hat{x}} \cdot \b x^0 (t)} \left[ -i\kappa_0  (\b{\hat{x}} \cdot \n (\b x^{0} (t))) \phi^h_0 (t) - \psi_0^h (t) \right] |\b x^{0'} (t)| dt .
\end{aligned}
\end{equation}

In the first example, we set the outer boundary curve $\Gamma_0$ to be a circle with center $(0,0)$ and radius $0.5,$ and the inner $\Gamma_1$ is a kite-shaped boundary  of the form
\[
\Gamma_1 = \left\lbrace \b x^1 (t) = (0.2 \cos t + 0.1 \cos 2t -0.2 , \, 0.2 \sin t + 0.1), \, t\in [0,2\pi]\right\rbrace .
\] 
We consider the points $\b z_1 = (-0.1, \, 0.35), \, \b z_2 = (0.1, \, 0.3) \in \Omega_1,$ and $\b z_3 = (-0.3, \, 0.55), \, \b z_4 = (0.15, \, 0.6) \in \Omega_0,$ see the left picture in \autoref{Fig2}. We set $\omega = 1, \, \lambda = 2$ and $\theta = \pi/3.$ The parameters are given by $(\epsilon_0 , \mu_0)= (1,1)$ and  $(\epsilon_1 , \mu_1)= (3,2).$

In \autoref{table1} we present the computed far-field of the electric and magnetic scattered fields at direction $t=0$ for increasing number of quadrature points $n.$ The exponential convergence is clearly exhibited, as we see also in \autoref{Fig3} where we plot the $L^2$ norm of the difference between the exact and the computed far-fields in logarithmic scale.

\begin{figure}[t]
\begin{center}
\includegraphics[scale=0.8]{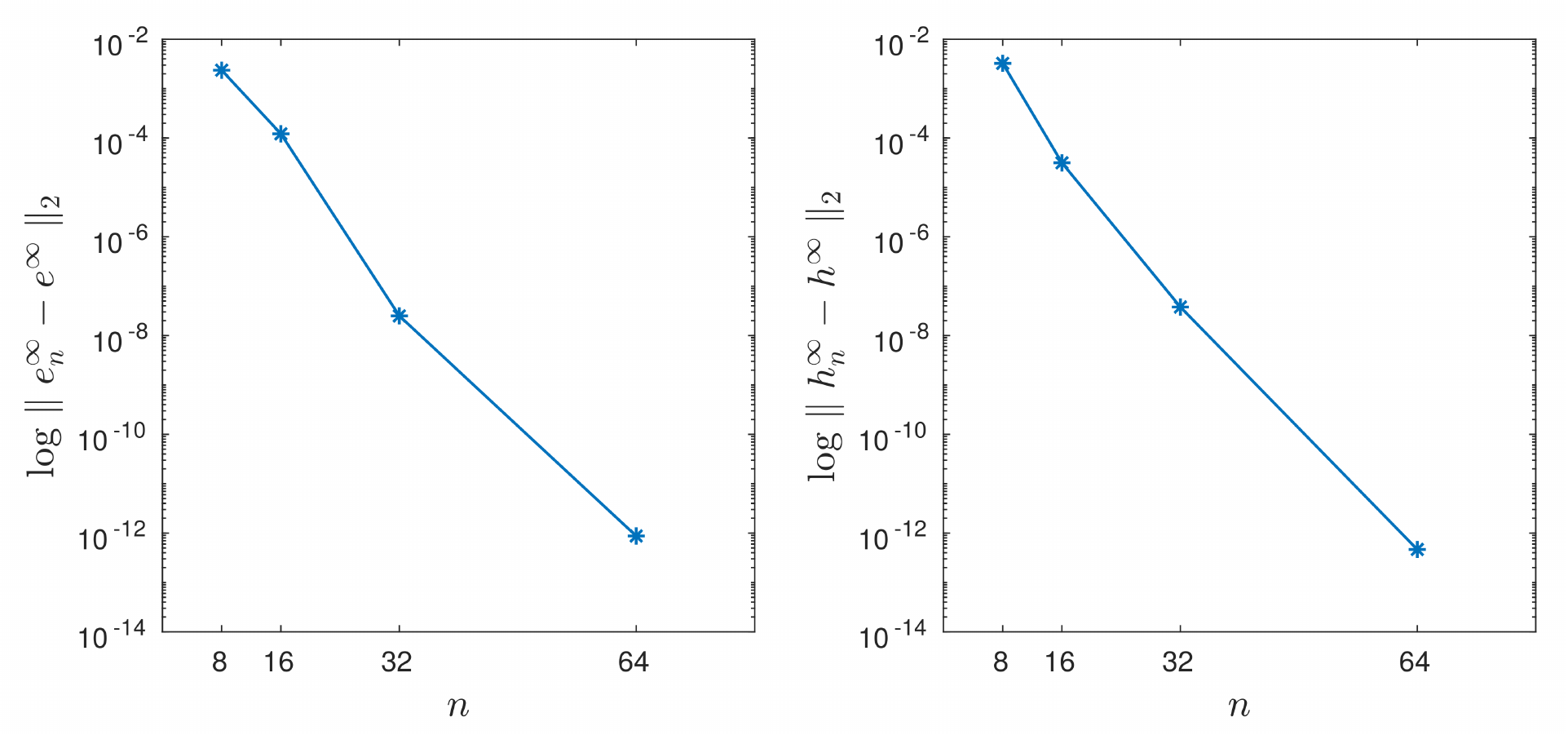}
\caption{The $L^2$ norm (in logarithmic scale) of the difference between the computed and the exact far-field of the electric (left) and the magnetic (right) scattered fields of the first example.  }\label{Fig3}
\end{center}
\end{figure}

In the second example, both boundaries admit the following parametrization
\[
\Gamma_j = \left\lbrace \b x^j (t) = r_j (t) (\cos t , \sin t ) + \b a^j,  \, t\in [0,2\pi]\right\rbrace , \quad j=0,1.
\] 
We consider a peanut-shaped outer boundary with radial function
\[
r_0 (t)= \left(0.5 \cos^2 t+ 0.1 \sin^2 t \right)^{1/2}, \quad \mbox{and} \quad \b a^0 = (0,0),
\]
and an apple-shaped inner boundary curve with
\[
r_1 (t) = \frac{0.45 + 0.3 \cos t -0.1 \sin 2t}{1+0.7 \cos t}, \quad \mbox{and} \quad \b a^1 = ( -0.25, \,  0.05).
\]

The source points are now: $\b z_1 = (0.2, \, 0.2), \, \b z_2 = (-0.5, \, -0.2) \in \Omega_1,$ and $\b z_3 = (0.4, \, 0.55),$ $\b z_4 = (-0.3, \, -0.6) \in \Omega_0,$ see the right picture in \autoref{Fig2}. We set $\omega = 2$ and $\theta = \pi/4,$ and we choose the parameters to be $(\epsilon_0 , \mu_0)= (2,1)$ and  $(\epsilon_1 , \mu_1)= (4,2).$ Here, the impedance function is given by
\[
\lambda (\b x^1 (t)) = \frac{1}{1 + 0.2 \cos t}.
\]

The computed far-field of the electric and magnetic scattered fields for increasing number of quadrature points $n,$ and the exact far-fields at direction $t = \pi/4$  are given in \autoref{table2}. Their $L^2$ norm difference is presented in \autoref{Fig4}. Again the exponential convergence is guaranteed independently of the different parameters. 

 \begin{table}[t]
\begin{center}
 \begin{tabular}{| c  | c  | c  | } 
 \hline
 $n$ & $ e_n^\infty (\b{\hat{x}}(\pi/4)) $ & $ h_n^\infty (\b{\hat{x}}(\pi/4)) $  
\\ \hline 
8 & $ 0.110995105311      - i \, 0.558817682095   $ & $  0.533445748875    + i \, 0.122603921042  $  \\ 
16 & $ 0.123115240894      -i \,0.551055700217   $ & $  0.552938285062     + i \,0.114751031182   $  \\
32 & $ 0.122965211004      -i \,0.550626647288   $ & $ 0.552427942024     + i \,0.114603756425   $ \\
64 & $  0.122964711410      -i \,0.550626521274   $ & $ 0.552427483456     + i \,0.114602625221   $ \\
\cline{1-1}\hhline{~==}
\multicolumn{1}{c|}{} &  $ e^\infty (\b{\hat{x}}(\pi/4)) $   & $h^\infty (\b{\hat{x}}(\pi/4))$ \\
\cline{2-3} 
\multicolumn{1}{c|}{} &  $ 0.122964711410      - i \,0.550626521275     $  & $  0.552427483455  + i \,   0.114602625221   $  \\ \cline{2-3}
\end{tabular}
\caption{The computed and the exact far-fields of the electric and magnetic scattered fields of the second example. }\label{table2}
\end{center}
\end{table}

\begin{figure}[t]
\begin{center}
\includegraphics[scale=0.85]{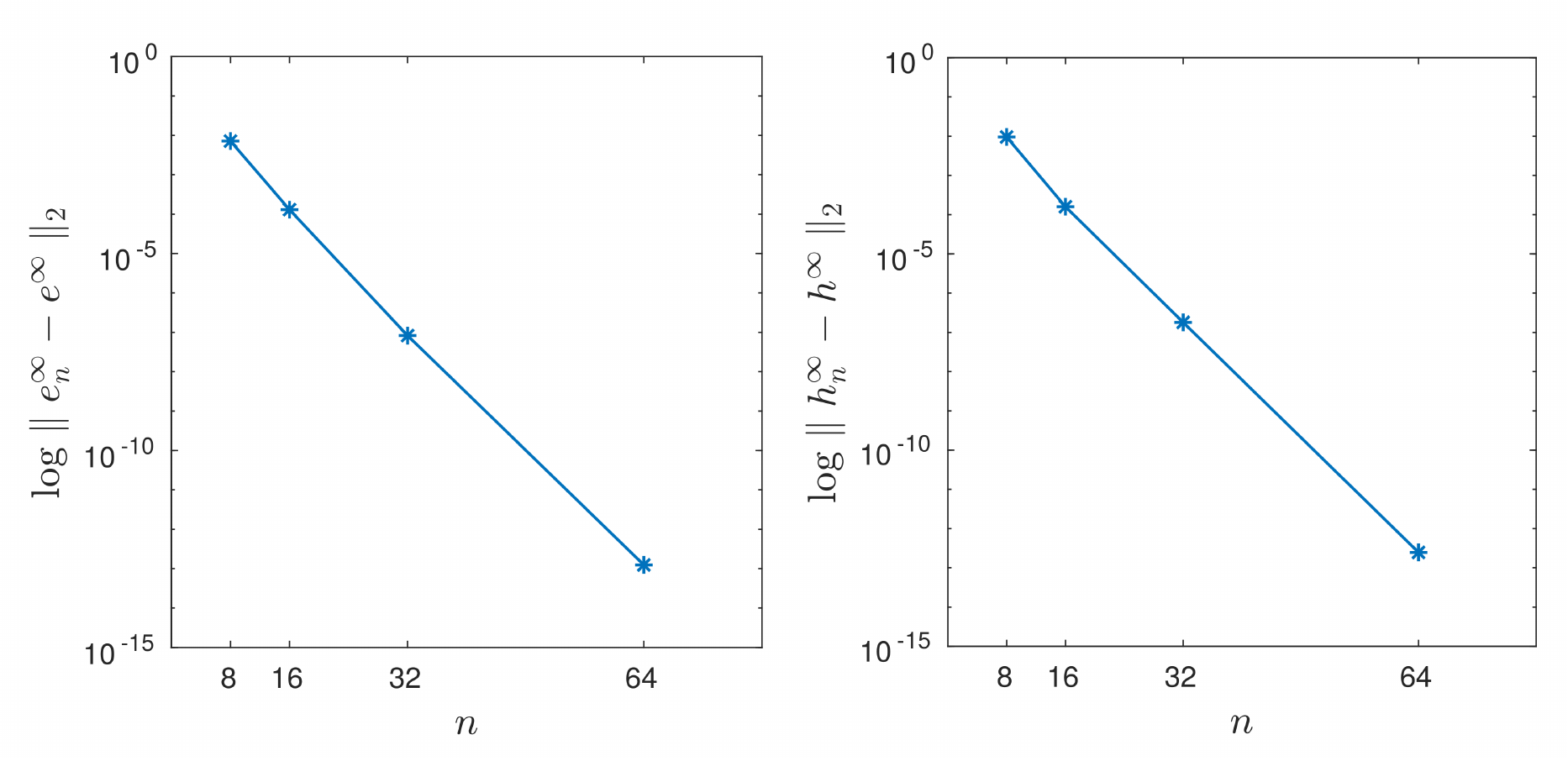}
\caption{The $L^2$ norm (in logarithmic scale) of the difference between the computed and the exact far-field of the electric (left) and the magnetic (right) scattered fields of the second example.  }\label{Fig4}
\end{center}
\end{figure}

\begin{figure}[h]
\begin{center}
\includegraphics[scale=0.75]{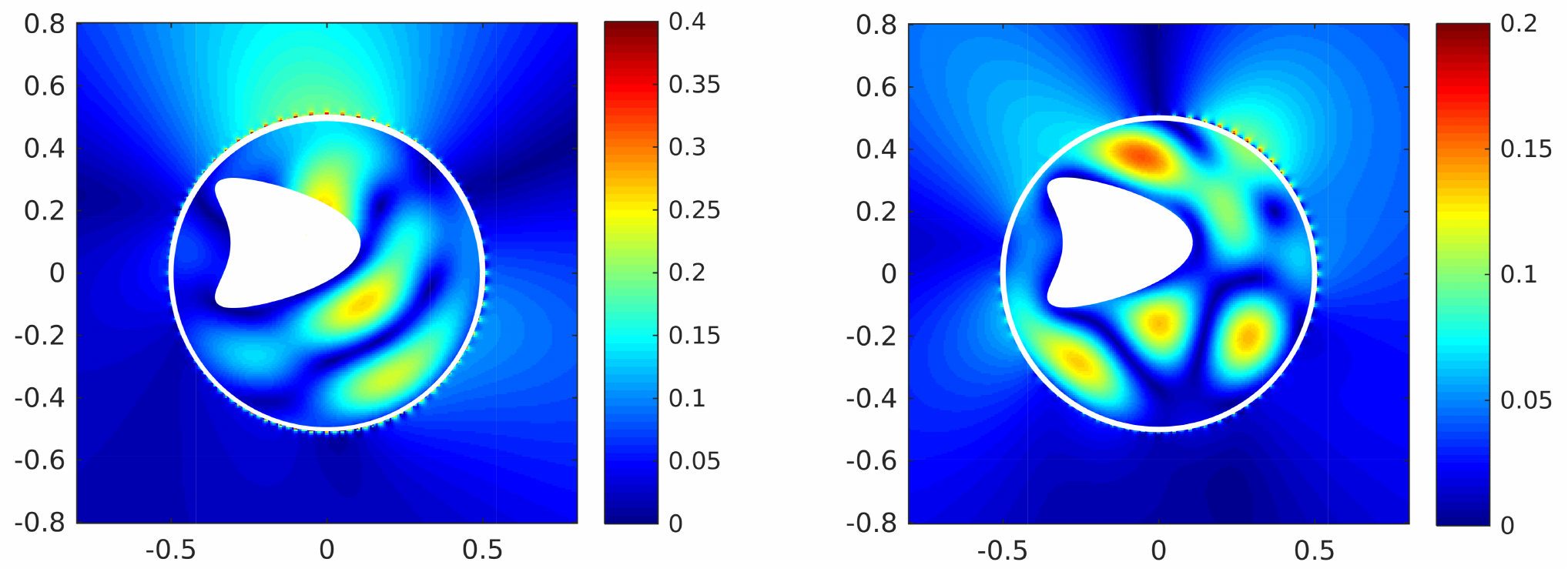}
\caption{The norms of the electric fields $e^0$ and $e^1$ (left) and those of the magnetic fields $h^0$ and $h^1$ (right) for $\phi = \pi/2.$ }\label{Fig5}
\end{center}
\end{figure}

\begin{figure}[ht]
\begin{center}
\includegraphics[scale=0.76]{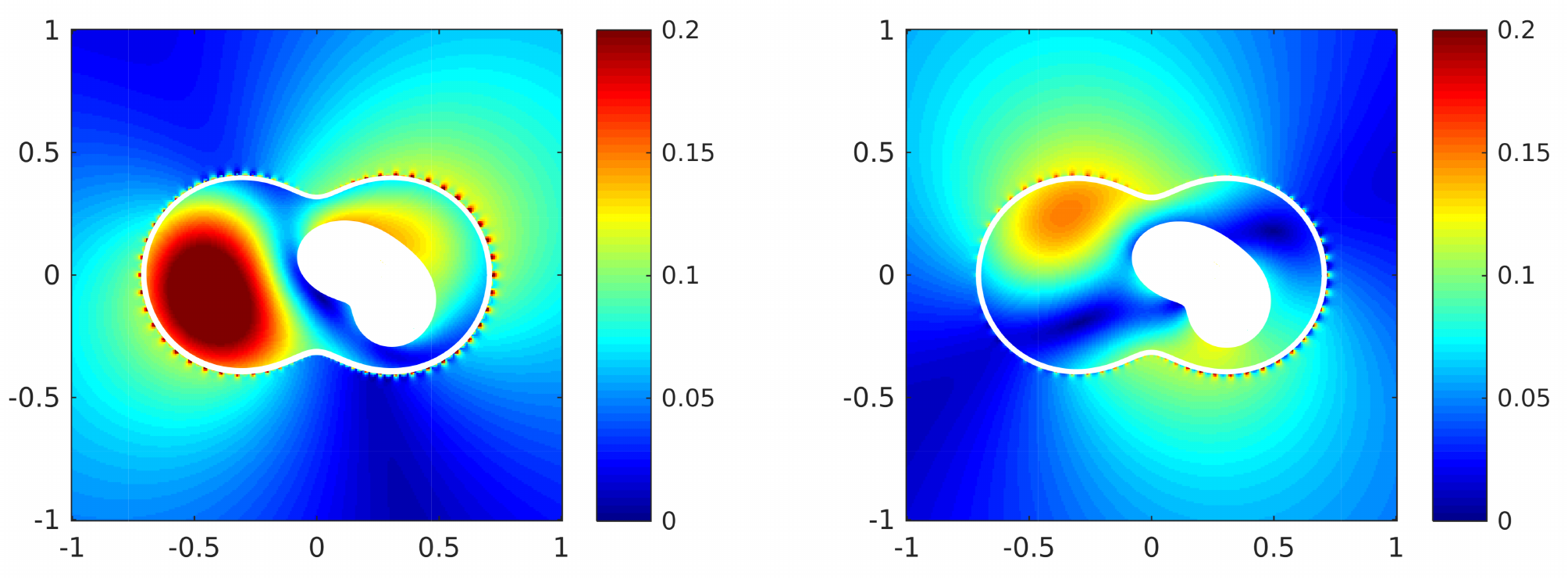}
\caption{The norms of the electric fields $e^0$ and $e^1$ (left) and those of the magnetic fields $h^0$ and $h^1$ (right) for $\phi = \pi/6.$ }\label{Fig6}
\end{center}
\end{figure}

\subsection{Examples with oblique incidence}

We consider obliquely incident waves of the form \eqref{incident} and 
we vary the polar angle $\phi,$ which corresponds to the incident direction $(\cos \phi, \, \sin \phi)$ in two dimensions. We restrict the computation domain to the rectangular domain $[-c,\,c]^2,$ where we consider a uniform-space grid of the form $\b x_{kj} = (-c + k \delta , -c + j\delta),$ with $\delta = 2c/(2m-1),$ for $k,j = 0,...,2m-1 .$ We use $m=128 .$

In the third example we consider the parametrizations of the first example and we set $\omega = 6$ and $\theta = \pi/4,$ while keeping all the other parameters the same. The values of the norms of the interior and exterior fields are given in \autoref{Fig5} for $c=0.8$ and $\phi = \pi/2.$

In the last example we consider the setup of the second example, where now $\b a^1 = (0.25,$ $ -0.05).$ We set $\omega = 1,$ and  $\theta = \phi = \pi/6.$ The parameters are given by $(\epsilon_0 , \mu_0)= (1,1)$ and  $(\epsilon_1 , \mu_1)= (6,4).$ In \autoref{Fig6} we see the results for $c=1.$

\section{Conclusions}

We addressed the direct electromagnetic scattering problem for a infinitely long, penetrable and doubly-connected cylinder. The cylinder was placed in a homogeneous dielectric medium and was illuminated by a time-harmonic electromagnetic wave at oblique incidence. We considered transmission conditions on the outer boundary and impedance boundary condition on the inner boundary. We proved the well-posedness of the problem using Green's formulas and the integral representation of the solutions (hybrid method). We presented numerical results which showed
the feasibility of the proposed method.

\section*{Acknowledgements}

The author would like to thank Drossos Gintides for discussions and his suggestions on this topic.

\end{document}